\documentclass[12pt,reqno]{amsart}
\usepackage{amssymb}
\usepackage{amsmath}
\usepackage{amsthm}
\usepackage{eucal}
\usepackage{color}
\usepackage{amsfonts}
\usepackage[pdftex]{graphicx}
\usepackage{mathtools}
\mathtoolsset{showonlyrefs}
\usepackage[normalem]{ulem}
\usepackage[T1]{fontenc}

\newcommand{\R}{\mathbb R}

\newcommand{\Z}{\mathbb Z}

\newcommand{\lcol}{\lfloor\!\!\!\lceil}
\newcommand{\rcol}{\rceil\!\!\!\rfloor}

\DeclareMathOperator{\sgn}{sgn}

\newtheorem{theorem}{Theorem}[section]

\newtheorem*{TA}{Theorem A}
\newtheorem*{TB}{Theorem B}
\newtheorem*{TC}{Theorem  C}
\newtheorem*{TD}{Theorem D}
\newtheorem*{TE}{Theorem E}
\numberwithin{equation}{section}
\theoremstyle{remark}
\newtheorem{remark}[theorem]{Remark}
\begin{document}
\title[The Camassa-Holm equation]{Unique Continuation Properties for solutions
to the \\
Camassa-Holm equation and related models.}
\author{Felipe Linares}
\address[F. Linares]{IMPA\\
Instituto Matem\'atica Pura e Aplicada\\
Estrada Dona Castorina 110\\
22460-320, Rio de Janeiro, RJ\\Brazil}
\email{linares@impa.br}

\author{Gustavo Ponce}
\address[G. Ponce]{Department  of Mathematics\\
University of California\\
Santa Barbara, CA 93106\\
USA.}
\email{ponce@math.ucsb.edu}


\begin{abstract} 
It is shown that if $\,u(x,t)\,$ is a real solution of the initial value problem for the Camassa-Holm equation which vanishes 
in an open set $\,\Omega\subset \R\times [0,T]$, then $\,u(x,t)=0,\,(x,t)\in\R\times [0,T]$. The argument of proof can be placed in a general setting to extend the above results to a class of non-linear non-local 1-dimensional models which includes the Degasperis-Procesi equation. This result also applies to solutions of the initial periodic boundary value problems associated to these models.

\end{abstract}

\thanks
{The first author was partially supported by CNPq and FAPERJ/Brazil.}

\maketitle

\section{Introduction}

This work is mainly concerned with  the  Camassa-Holm (CH) equation on the real line
\begin{equation}\label{CH1}
\partial_tu+3 u \partial_xu-\partial_t \partial_x^2 u=2\partial_xu \partial_x^2u+ u\partial_x^3u, \hskip5pt \;t,\,x\in\R.
\end{equation}

The CH equation \eqref{CH1} was first noted by Fuchssteiner and Fokas \cite{FF} in their work on hereditary symmetries. Later, it was written explicitly and derived physically as a model for shallow water waves  by Camassa and Holm \cite{CH}, who also examined its solutions. It also appears as a model in nonlinear dispersive waves in hyper-elastic rods \cite{Dai}, \cite{DaHu}. 

The CH equation \eqref{CH1} has received extensive attention due to its remarkable properties, among them the fact that it is a bi-Hamiltonian completely integrable model  (see \cite{BSS},  \cite{CH}, \cite{CoMc}, \cite{Mc}, \cite{M}, \cite{Par} and references therein).

 The CH equation possesses ``peakon'' solutions
\cite{CH}. In the case of a single peakon this solitary wave solution can be written as
\begin{equation}\label{peakon}
u_c(x,t)=c\, e^{-|x-ct|}, \hskip20pt c>0.
\end{equation}
The multi-peakon solutions exhibit the  ``elastic'' collision property that reflect their soliton character (see \cite{BSS2}).

It is convenient to write the CH equation \eqref{CH1} in the following (formally) equivalent form

\begin{equation}\label{CH}
\partial_tu+ u \partial_xu+\partial_x(1-\partial_x^2)^{-1}\big(u^2+\frac{1}{2} (\partial_xu)^2\big)=0, \hskip5pt \;t,\,x\in\R.
\end{equation}

The initial value problem (IVP)  as well as the initial periodic boundary value problem (IPBVP)  associated to the equation \eqref{CH} has been extensively examined. 
In particular, in \cite{LO} and \cite{RB}  strong local well-posedness (LWP) of the IVP was established in the Sobolev space
\[
H^s(\R)=(1-\partial_x^2)^{-s/2}L^2(\R),\;\;\;\;\;\;s>3/2.
\]
The peakon solutions do not belong to these spaces. In fact, 
\begin{equation}
\label{space}
\phi(x)=e^{-|x|}\notin W^{p,1+1/p}(\R)\;\;\;\;\;\text{for any}\;\;\;\;\;p\in[1,\infty),
\end{equation}
where $W^{s,p}(\R)=(1-\partial_x^2)^{-s/2}L^p(\R)$ with $\,W^{s,2}(\R)=H^s(\R)$.

However,
\[
\phi(x)=e^{-|x|}\in W^{1,\infty}(\R),
\]
where $W^{1,\infty}(\R)$ denotes the space of Lipschitz functions.

In \cite{CoEs1} it was proved that if $u_0\in H^1(\R)$ with $u_0-\partial_x^2u_0\in \mathcal {M}^+(\R)$, where $\mathcal{M}^+(\R)$ denotes the set of positive Radon measures with bounded total variation, then the IVP for the CH equation \eqref{CH} has a global weak solution $u\in L^{\infty}((0, \infty):H^1(\R))$.

 An improvement of the previous result was obtained in  \cite{CoMo} by showing that if $u_0\in H^1(\R)$ with $u_0-\partial_x^2u_0\in \mathcal {M}^+(\R)$,
then the  IVP for the CH equation \eqref{CH} has a unique solution 
$$
u\in C([0,\infty):H^1(\R))\cap C^1((0,\infty):L^2(\R))
$$
 satisfying that
$ y(t)\equiv u(\cdot,t)-\partial_x^2u(\cdot,t)\in \mathcal {M}^+(\R)$ is uniformly bounded in $[0,\infty)$.

In \cite{XZ}  the existence of a $H^1$-global weak solution for the IVP for the CH equation \eqref{CH} for 
data $u_0\in H^1(\R)$ was established. 

In \cite{CoEs1} and \cite{CoEs2} (see also \cite{LO}) there were deduced conditions on the data $u_0\in H^3(\R)$ which guarantee that the
corresponding  local solution $u\in C([0,T]:H^3(\R))$ of the IVP associated to the CH \eqref{CH} blows up in finite time
by showing that
$$
\lim_{t\uparrow T}\|\partial_xu(\cdot,t)\|_{\infty}=\infty,
$$
corresponding to the breaking of waves.
Observe that $H^1$-solutions of the CH equation \eqref{CH1} satisfy the conservation law
$$
E(u)(t)=\int_{-\infty}^{\infty}(u^2+(\partial_xu)^2)(x,t)dx= E(u_0),
$$
so that the $H^1$-norm of the solutions  remains invariant 
within the existence interval.

More recently, in \cite{BC} and \cite{BCZ}  the existence and uniqueness, 
 respectively,  of a $H^1$ global solution for the CH equation \eqref{CH} was  settled.

 For other well-posedness results see also \cite{GHR1}, \cite{GHR2}  and references  therein.
\vskip.1in

We shall describe the class of solutions which we will be working with.
 First, we consider the IVP and recall a result found in \cite{LiPoSi} motivated by an early work \cite{LKT} for the IPBVP:

\begin{TA}[\cite{LiPoSi}]\label{thm1}
Given  $u_0\in X\equiv H^1(\R)\cap W^{1,\infty}(\R)$, there exist a non-increasing function $T=T(\|u_0\|_X)>0$ and a unique solution $u=u(x,t)$ of the IVP associated to the CH equation \eqref{CH}
such that
\begin{equation}
\begin{aligned}
\label{class-sol}
u\in Z_T
\equiv &C([0,T]\!:\!H^1(\R))\cap C^1((0,T)\!:\!L^2(\R)\\
&
\cap C^1((0,T)\!:\!L^2(\R))= \mathcal Y_T \cap C^1((0,T)\!:\!L^2(\R),
\end{aligned}
\end{equation}
with
\[
\sup_{[0,T]}\|u(\cdot,t)\|_X=\sup_{[0,T]}(\|u(\cdot,t)\|_{1,2}+\|u(\cdot,t)\|_{1,\infty})\leq c\|u_0\|_X,
\]
for some universal constant $c>0$.
Moreover, given $R>0$, the map $u_0\mapsto u$, taking the data to the solution, is continuous from the ball $\{u_0\in X :\|u_0\|_X\le R\}$
into $\mathcal Y_{T(R)}$.
\end{TA}

\begin{remark}\label{rem1}
 The strong notion of LWP introduced in \cite{Ka} does not hold in this case. 
This notion includes existence, uniqueness, {\it persistence property}, namely that
if $u_0\in Y$, then $u\in C([0,T]:Y)$, and that the map data $\,\mapsto\,$ solution is locally continuous from $Y$ to $C([0,T]:Y)$. In particular, this strong version of
LWP  guarantees that the solution flow defines a dynamical system in $Y$. 

As it was mentioned before the strong concept of LWP holds in $\,H^s(\R)$ with $s>3/2$, where the peakon solutions are not included. 
  \end{remark}
\vskip.1in

In this work we are interested in unique continuation properties of solutions of the CH equation. Thus, we recall two results established in \cite{HMPZ} 
regarding unique continuation and decay persistence properties of solutions of the IVP for the equation \eqref{CH}:

\vskip.1in

\begin{TB}[\cite{HMPZ}] \label{A}  Assume that for some $T>0$ and $s>3/2$,
\[
u\in C([0,T]: H^s(\R))\cap C^{1}((0,T):H^{s-1}(\R))
\]
is a strong real solution of the IVP associated to the CH equation \eqref{CH}. If for some $\alpha\in(1/2,1)$, $u_0(x)=u(x,0)$ satisfies  
\begin{equation}\label{h1}
|u_0(x)|= o(e^{-x})\;\;\;\;\text{and}\;\;\;\;|\partial_xu_0(x)|=  O(e^{-\alpha x}),\;\;\;\;\text{as}\;\;\;x\uparrow \infty,
\end{equation}
and there exists $t_1\in (0,T]$ such that 
\[
|u(x,t_1)|= o(e^{-x}),\;\;\;\;\;\;\text{as}\;\;\;x\uparrow \infty,
\]
then $u\equiv 0$.
\end{TB}

 Roughly, Theorem B is optimal:

\begin{TC}[\cite{HMPZ}] \label{B}  Assume that for some $T>0$ and $s>3/2$,
\[
u\in C([0,T]: H^s(\R))\cap C^{1}((0,T):H^{s-1}(\R))
\]
is a strong real solution of the IVP associated to the CH equation \eqref{CH}. If 
for some $\theta\in(0,1)$,
$u_0(x)=u(x,0)$ satisfies 
\[
|u_0(x)|,\;\;\;\;|\partial_xu_0(x)|=  O(e^{-\theta x}),\;\;\;\;\text{as}\;\;\;x\uparrow \infty,
\]
then
\[
|u(x,t)|,\;\;\;\;|\partial_xu(x,t)|=  O(e^{-\theta x}),\;\;\;\;\text{as}\;\;\;x\uparrow \infty,
\]
uniformly in the time interval $[0,T]$.
\end{TC}

\begin{remark} In \cite{LiPoSi}  Theorem B and Theorem C were extended to the class considered in Theorem A.

\end{remark}

\vskip.1in

Our first result in this work is :

\begin{theorem} \label{IVPCH}
Let $\,u=u(x,t)$ be a real solution of the IVP associated to the CH equation \eqref{CH} in the class described in Theorem A.
If there exists an open set $\,\Omega\subset \R\times [0,T]$ such that 
\begin{equation}
\label{cond1}
u(x,t)=0,\;\;\;\;\;\;\;(x,t)\in \Omega,
\end{equation}
then $\,u\equiv 0$.

\end{theorem}

\begin{remark} (i) To the best of our knowledge Theorem \ref{IVPCH} is the only unique continuation result available for solutions of the IVP associated to the CH equation.

(ii) A stronger version of this result has been established for the Korteweg-de Vries (KdV) equation,  the1-dimensional non-linear Schr\"odinger (NLS) equation and the Benjamin-Ono (BO) equation in \cite{SaSc}, \cite{Iz} and \cite{KPV19} respectively. More precisely, it was proven there that if $\, u_1,\,u_2$ are two solutions of these equations which agree in an open set $\,\Omega\subset \R\times [0,T]$, then they are identical.

 Our approach is simpler than the ones  in  \cite{SaSc}, \cite{Iz} and \cite{KPV19} but it does not allow us to obtain the above mention stronger result. It applies only to a single solution of the CH equation since it depends on the whole structure of the equation. 

Roughly speaking, this is related to the unique continuation property known for these equations under assumptions of  decay at infinity at two different times. For the KdV and the 1-dimensional NLS equations results are known for the difference of two solutions $u_1,\,u_2$, see \cite{EKPV}, \cite{EKPV2} and references therein.  However, the corresponding results for the BO and the CH equations require that $u_2(x,t)\equiv 0$, see  \cite{FLP}, \cite{LP} and references therein for the BO equation and  \cite{HMPZ} for the CH equation.
 
\end{remark}

\begin{remark}
 In the proof of Theorem \ref{IVPCH} the only condition on the structure of integrand term in \eqref{CH} 
 $$
\big( u^2+\frac{1}{2}\,(\partial_xu\big)^2)(x,t) 
 $$
 needed  is that it is non-negative. Hence, the same proof provides a similar result for any equation of the form
 \begin{equation}
 \label{gen1}
 \partial_tu+g(u,\partial_xu)+\partial_x(1-\partial_x^2)^{-1} h(u,\partial_xu)=0,
 \end{equation}
 with
 \begin{equation}
 \label{gen1a}
 g(\cdot,\cdot),\,\,h(\cdot,\cdot)\;\;\;\;\;\text{ smooth}\;\;\;\;\;\; g(0,0)=h(0,0)=0,
 \end{equation}
  and 
 \begin{equation}
 \label{gen2}
 h(x,y)>0,\;\;\;\;\;\;\;\;\;\forall \,(x,y)\neq (0,0).
 \end{equation}
 
 If $h(\cdot,\cdot)=h(\cdot)$  one requires that $h(x)>0$ whenever $x\neq 0$.

The class of equations described in \eqref{gen1}-\eqref{gen2} includes  the so called  $b$-equations (rod equations), see \cite{EY}, \cite{DaHu},
$$
\partial_tu+(b+1) u \partial_xu-\partial_t \partial_x^2 u=b\partial_xu \partial_x^2u+ u\partial_x^3u
$$
which can be written as 
\begin{equation}
\label{b-eq}
\partial_tu+u\partial_x u +\,\partial_x(1-\partial_x^2)^{-1}\Big(\,\frac{b}{2}u^2  +\frac{3-b}{2} (\partial_xu)^2\Big)=0,\;\;\;\;\;b\in[0,3].
\end{equation}

The parameter $b$ is related to the Finger deformation tensor to the material of the rod, see \cite{DaHu}.

Notice that for $b=2$ in \eqref{b-eq} one gets the CH equation meanwhile for $b=3$ in \eqref{b-eq} one obtains the Degasperis-Procesi (DP) equation \cite{DP}, the only bi-hamiltonian and integrable models in this family, see \cite{Iv}. Thus, the DP model possesses peakon solutions, see \eqref{peakon}, which display elastic collision properties, \cite{Ma}. 
\end{remark}

Therefore, we have :

\begin{theorem}
\label{general}
The result in Theorem \ref{IVPCH} applies to any real solution $\,u(\cdot,\cdot)$ in the class \eqref{class-sol} of the IVP associated to the equation \eqref{gen1}  satisfying the  hypotheses \eqref{gen1a} and \eqref{gen2}.

In particular, it covers all the $b$-equations described in \eqref{b-eq}.

\end{theorem}

Finally, we shall consider the initial periodic boundary value problem  (IPBVP) associated to the above models. 

In this periodic setting one finds the following unique continuation results for the $b$-equations established in \cite{BrCo}, 

\begin{TD}[\cite{BrCo}] \label {BrCo}
Let $u\in C([0,\infty]\!:\!H^s(\mathbb S))\cap C^1((0,\infty)\!:\!H^{s-1}(\mathbb S)),\,s>3/2,$ be a global real solution to the IPBVP for $b$-equations \eqref{b-eq}
with $0\leq b\leq 2.731$. If $u$ vanishes at some point $(x_0,t_0)\in \mathbb S\times [0,\infty)$, then $u\equiv 0$.
\end{TD}

\begin{remark} Theorem D improves an earlier result in \cite{CoEs3} for the CH equation $b=2$ where it was assumed that the global 
solution satisfies that for any $t\in[0,\infty)$ there exists $x_t\in\mathbb S$ such that $u(x_t,t)=0$.
\end{remark} 

Analogous LWP results to those in Theorem A for the IVP were previously obtained in \cite{LKT} for the IPBVP.   More precisely:

\begin{TE}[\cite{LKT}] \label{thm1p}
Given  $u_0\in \mathcal X\equiv H^1(\mathbb S)\cap W^{1,\infty}(\mathbb S)$, there exist a non-increasing function $T=T(\|u_0\|_{\mathcal X})>0$ and a unique solution $u=u(x,t)$ of the IPBVP associated to the CH equation \eqref{CH}
such that
\begin{equation}
\begin{aligned}
\label{class-solp}
u\in Y_T
\equiv &C([0,T]\!:\!H^1(\mathbb S))\cap C^1((0,T)\!:\!L^2(\mathbb S))\\
&\cap L^{\infty}([0,T]\!:\!W^{1,\infty}(\mathbb S))=\mathcal S_T\cap L^{\infty}([0,T]\!:\!W^{1,\infty}(\mathbb S)),
\end{aligned}
\end{equation}
with
\[
\sup_{[0,T]}\|u(\cdot,t)\|_{\mathcal X}=\sup_{[0,T]}(\|u(\cdot,t)\|_{1,2}+\|u(\cdot,t)\|_{1,\infty})\leq c\|u_0\|_{\mathcal X},
\]
for some universal constant $c>0$.
Moreover, given $R>0$, the map $u_0\mapsto u$, taking the data to the solution, is continuous from the ball 
$\{u_0\in \mathcal X :\|u_0\|_{\mathcal X}\le R\}$
into $\mathcal{S}_{T(R)}$.
\end{TE}

Our next theorem states that the results in Theorem \ref{general} for the IVP also hold for the IPBVP for the models in  \eqref{gen1} under the assumptions \eqref{gen1a}-\eqref{gen2}. In particular, includes all the $b$-equations in \eqref{b-eq}.

\begin{theorem} \label{IPBVPCH}
The result in Theorem \ref{IVPCH} applies to any solution $\,u(\cdot,\cdot)$ in the class \eqref{class-sol} of the IPBVP associated to the equation \eqref{gen1}  satisfying the  hypotheses \eqref{gen1a} and \eqref{gen2}.

In particular, it covers all the equations described in \eqref{b-eq}.

\end{theorem}

\begin{remark} It is interesting to compare the results in Theorem D and Theorem \ref{IPBVPCH}. On one hand, the vanishing assumption on the former is 
significantly weaker than that in the latter. On the other hand,   Theorem D (which only applies to the IPBVP) does not cover  all the $b$-equations in \eqref{b-eq}. In particular, it does 
not apply to the DP model ($b=3$). Moreover, since the equation is time reversible it is clear that the result in Theorem D does not extend to solutions 
of the IPBVP having finite life span.

Theorem \ref{IPBVPCH} applies to  all the $b$-equations in \eqref{b-eq}  and to any local solution of the IPBVP associated with the equation \eqref{gen1}.
\end{remark}

The rest of this work is organized as follows: Section 2 contains the proofs of Theorem \ref{IVPCH} and Theorem \ref{IPBVPCH}. It is also shown how the
argument in the proofs   of Theorem \ref{IVPCH} and Theorem \ref{IPBVPCH} can be extended to prove Theorem \ref{general}.
\section{Proof of Theorem \ref{IVPCH} and Theorem \ref{IPBVPCH}}



First we shall prove Theorem \ref{IVPCH}.
\vskip.1in
\begin{proof} [Proof of Theorem \ref{IVPCH}]

We recall that 
\begin{equation}
\label{fund-sol}
(1-\partial_x^2)^{-1}h(x)=\frac{1}{2}\,\big(e^{-|\cdot| }\ast h \big)(x),\;\;\;\;\;\;h\in L^2(\R).
\end{equation}

From  the hypothesis it follows that
\begin{equation}\label{P1b}
 \,\Big(u^2+\frac{(\partial_xu)^2}{2}\Big) \Big|_{\Omega}\equiv 0,
\end{equation}
 and from the equation \eqref{CH}  one gets 
\begin{equation}
\label{P1}
\partial_x(1-\partial_x^2)^{-1}\Big(u^2+\frac{(\partial_xu)^2}{2}\Big) \Big|_{\Omega}\,\equiv 0.
\end{equation}

Thus, $\;\exists\, t^*\in (0,T)$ and $\,I=[a,b],\;a<b$, $[a,b]\times\{t^*\}\subset \Omega$ such that defining
\begin{equation}
\label{P2}
\begin{aligned}
F(x)&:=\partial_x(1-\partial_x^2)^{-1}\Big(u^2+\frac{(\partial_xu)^2}{2}\Big)(x,t^*)\\
&=-\frac{1}{2}\,\sgn(\cdot)\,e^{-|\cdot|} \ast \Big(u^2+\frac{(\partial_xu)^2}{2}\Big)(x,t^*)
\end{aligned}
\end{equation}
and
\begin{equation}
\label{P2a}
f(x):=(u^2+\frac{(\partial_xu)^2}{2})(x,t^*)
\end{equation}
one has that
\begin{equation}
\label{P3}F\in L^1(\R)\cap L^{\infty}(\R)\cap C(\R),\;\;\;f\in L^1(\R)\cap L^{\infty}(\R),
\end{equation}
with
\begin{equation}
\label{P3a}
F(x)=f(x)=0,\;\;\;\;\;\;\,\;\;\;\;x\in [a,b].
\end{equation}
\vskip.05in

We observe that for any $y\notin [a,b]$
\begin{equation}
\label{P4}
-\sgn(b-y)\,e^{-|b-y|}> -\sgn(a-y)\,e^{-|a-y|}.
\end{equation}
Hence,
\begin{equation}
\label{P5}
\begin{aligned}
F(b)&=-\frac{1}{2}\,\int_{-\infty}^{\infty}\sgn(b-y) e^{-|b-y|}f(y)\,dy\\
&\geq -\ \frac{1}{2}\,\int_{-\infty}^{\infty}\sgn(a-y) e^{-|a-y|}f(y)\,dy=F(a),
\end{aligned}
\end{equation}
with $\,f\ge 0$ and 
\begin{equation}
\label{P6}
F(b)=F(a) \;\;\;\;\;\text{if and only if }\;\;\;\;f\equiv 0.
\end{equation}

Since, $F(b)=F(a)=0, $ we obtain the desired result.

\end{proof}

\begin{remark} 
One can give a different proof by showing that $\,F(\cdot)\,$ defined in \eqref{P2} is differentiable in $\,(a,b)$, (see \cite{LiPoSi}),
 with
 $$
 F'(x)= (e^{-|\cdot|}\ast f)(x) -f(x),\;\;\;\;\;x\in (a,b),
 $$
 (since $\,\partial_x^2(1-\partial_x^2)^{-1}=(1-\partial_x^2)^{-1}-1$). 
Therefore, since $f(x)=0$ for any $\;x\in [a,b]$ and $\,f\geq 0$ on $\,\R$ one has that
 $$
 F'(x)\geq 0,\;\;\;x\in (a,b),
  $$
 with
 $$
 F'(x)=0 \;\;\text{if and only if}\;\;f\equiv 0.
 $$
 
 Recalling \eqref{P3a}, i.e.   $\,F(a)=F(b)=0,\,$ one gets the result.
 \vskip.2in
 
 \end{remark}
\begin{proof}[Proof of Theorem \ref{IPBVPCH}]
The proof  is similar to that given for the IVP in Theorem \ref{IVPCH}. The only difference is to show that the equivalent  inequality in \eqref{P4} is satisfied in $\mathbb S \simeq \R/\Z\simeq [0,1)$. 

We recall that if $\,h\in L^2(\mathbb S)$, then 
$$
\partial_x(1-\partial_x^2)^{-1}h(x)= (\partial_x G\ast h)(x)
$$
where
\begin{equation}
\label{green}
G(x)=\frac{ \cosh(x-\lcol x\rcol-1/2)}{2\,\sinh(1/2)},\;\;\;\;x\in\R,
\end{equation}
and $\lcol \cdot\rcol\,$ denotes the greatest integer function. Observe that $G$ is differentiable in $\R-\Z$.

Thus, here $\,G(x)\,$ plays the role of the (Green) function $\,e^{-|x|}/2$ on the line (for $\,(1-\partial_x^2)$).

\vskip.1in
Hence, to obtain the equivalent expression to \eqref{P4} one has to show: if $0<a<b<1$, then
\begin{equation}
\label{007}
\partial_xG(b-y)>\partial_xG(a-y),\;\;\;\;\;y\in[0,1]-[a,b].
\end{equation}

Since
$$
\partial_xG(x)=\frac{ \sinh(x-\lcol x\rcol-1/2)}{2\,\sinh(1/2)},
$$
\vskip.07in
\noindent it suffices to see that if $\,y\in[0,1]-[a,b]$, then
$$
\sinh\big(b-y-\lcol b-y\rcol-\frac12\big)>\sinh\big(a-y-\lcol a-y\rcol -\frac12\big).
$$
 By combining that:
 \begin{equation}
 \label{last}
 \begin{aligned}
 \text{if}&\;\;\;\;y\in[0,a],\;\;\;\;\text{then}\;\;\;\,\lcol b-y\rcol=\lcol a-y\rcol=0,\\
 \text{if}&\;\;\;\;y\in[b,1],\;\;\;\;\,\text{then}\;\;\; \,\lcol b-y\rcol=\lcol a-y\rcol=-1,
 \end{aligned}
 \end{equation}
 and the fact that $\,\sinh(\cdot)$ is strictly increasing the proof is concluded.
\end{proof}

\vskip1mm

The proof of Theorem \ref{general} will be omitted since the argument follows same lines as the proofs of Theorem \ref{IVPCH} and
Theorem \ref{IPBVPCH} given in detail above.

\end{document}